\documentclass[12pt, reqno]{amsart}
\usepackage[utf8]{inputenc}
\DeclareUnicodeCharacter{21AF}{\downzigzagarrow}
\usepackage{amssymb,amsmath,latexsym}
\usepackage[varg]{pxfonts}
\usepackage{amsthm, amscd, amsfonts, amssymb, graphicx, color}
\usepackage[bookmarksnumbered, colorlinks, plainpages]{hyperref}
\usepackage{marvosym}
\usepackage{paralist}
\usepackage{color}
\usepackage{vruler}
\usepackage{hyperref}
\usepackage{cleveref}
\theoremstyle{plain}
\newtheorem{lem}{\textbf{Lemma}}[section]
\newtheorem{theorem}[lem]{\textbf{Theorem}}

\newtheorem{definition}[lem]{\textbf{Definition}}

\newtheorem{example}[lem]{\textbf{Example}}
\newtheorem{proposition}[lem]{\textbf{Proposition}}
\newtheorem{remark}{\textbf{Remark}}
\theoremstyle{definition}

\theoremstyle{remark}

\newtheorem*{conclusion}{\textbf{Concluding Remarks}}

\begin{document}

\thispagestyle{plain}
{\noindent Journal of Mathematics \\
Vol. XX, No. XX, (2022), pp-pp (Will be inserted by layout editor)}\\
ISSN: 2314-4629 (Print)
ISSN: 2314-4785 (Online)\\
URL: http://www.hindawi.com/journals/jmath\\

\noindent {}   \\[0.50in]
\noindent {}   \\[0.50in]


\title[ ]{ON RADICAL OF INTUITIONISTIC FUZZY PRIMARY SUBMODULE}

\author[  Taherpour, Ghalandarzadeh, Malakooti Rad, Safari]{Abbas Taherpour$^{1}$, Shaban Ghalandarzadeh$^{2,*}$,  Parastoo Malakooti Rad$^{3}$,  Parvin Safari$^{4}$}
\thanks{$^{*}$ Corresponding author\\ 2020 Mathematics Subject Classification. 08A72; 03F55.}

\address{$^{1,3,4}$Department of Mathematics, Qazvin Branch, Islamic Azad University, Qazvin, Iran}
\address{$^{2}$Faculty of Mathematics, K.N.Toosi University of Technology, Tehran, Iran}

\email{ataherpour437@gmail.com(A. Taherpour); ghalandarzadeh@kntu.ac.ir(S. Ghalandarzadeh);
pmalakoti@gmail.com(P. Malakooti Rad); parvin.safari@gmail.com(P.Safari)}



\keywords{}
\maketitle
\hrule width \hsize \kern 1mm

\begin{abstract}
In this paper, we further study the theory of Intuitionistic fuzzy submodules and we will define intuitionistic fuzzy primary submodule with the help of the definition of a radical submodule, and we also study the properties of these submodules. Furthermore homomorphic image and pre-image of intuitionistic fuzzy primary submodule are investigated. 

\textbf{Keywords}: intuitionistic fuzzy primary submodule, Intuitionistic fuzzy weakly primary ideal, Intuitionistic fuzzy ideal.
\end{abstract}
\maketitle
\vspace{0.1in}
\hrule width \hsize \kern 1mm

\section{\textbf{Introduction}}
Throughout this paper, $ R $ is a commutative ring with unity and $ M $ is an unitary R-module. A proper submodule $ P $ of $ M $ is called prime submodule provided that for any $ r\in R $ and $ m\in M $, $ rm\in P $ implies that $ m\in P $ or $ r\in (P:M)=\lbrace r\in R \vert rM\subseteq P\rbrace $. The set of all prime submodules of $ M $ is called prime spectrum of $ M $ and is denoted by $ Spec(M) $. The Weakly prime submodules have been introduced by Behboodi et al, in \cite{11} . Also, the spectrum of weakly prime submodule has been investigated in \cite{7} . Following this, the concepts of primal submodules over the noncommutative ring and nearly prime submodules have been presented in \cite{1,13}.
 
One of the structures that play an important role in mathematics is the structure of fuzzy algebra, which has been used extensively in many fields such as computer science, information technology, theoretical physics, control engineering, etc. Since the introduction of fuzzy sets in 1965 \cite{10}, researchers have done extensive research on various concepts of abstract algebra in the space of fuzzy sets, such as Rosenfeld\cite{2}, who in 1971 was the first to define the concept of fuzzy subgroups of a group. Many extensions of this concept have been proposed since then (especially in recent decades). In 1975, Relescu, Naegoita\cite{4} applied the concept of fuzzy sets to the theory of modules. After presenting this definition, different types of fuzzy submodules were examined in the last two decades. In 1986, Atanassov\cite{8} introduced intuitionistic fuzzy sets based on the degree of membership and degree of non-membership, provided that the total membership and non-membership should not exceed one another. In 1989, Biswas\cite{16} applied the concept of intuitionistic fuzzy sets to group theory and studied the intuitionistic fuzzy subsets of a group. In the last few years a considerable number of papers have been done on fuzzy and intuitionistic fuzzy submodules in general, and fuzzy and intuitionistic fuzzy prime and primary submodule in particular.
Hur et al.\cite{6} introduced the notion of intuitionistic fuzzy prime ideals and intuitionistic fuzzy weakly completely prime ideals in a ring. The concept of the fuzzy prime submodule and fuzzy primary submodule was studied by Mashinchi and Zahedi in \cite{12}. Also intuitionistic fuzzy submodules and their properties were studied by many mathematicians(\cite{3},\cite{18},\cite{14}). In this paper, and in the first step we define the intuitionistic fuzzy radical of an intuitionistic fuzzy submodule of an R-module $M$ and so we present some of its properties similar to the studies in the module theory. Also and related to this definition, we will define intuitionistic fuzzy primary submodule and investigate their properties. We show that for any intuitionistic fuzzy primary submodule of an R-module $M$, with sup property its radical will be an intuitionistic fuzzy prime ideal of $R$. \\

\section{\textbf{Definitions and Preliminary Results}}
In this paper, all rings are commutative with unity$(1\neq 0)$ and all R-modules are unitary. We use the symbol $ \theta $ for the zero element of a R-module. Let $ X $ be a non-empty set. An intuitionistic fuzzy subset is a function $ A=(\mu_A,\nu_A)$, which $ \mu_A(x) $ denotes the degree of membership and $ \nu_A(x) $ denotes the degree of non-membership of $ x \in X $ to subset $ A $, such that $0 \leq \mu_A(x) + \nu_A(x) \leq 1$ for all $ x \in X $, can be denoted by $A = \lbrace x ,\mu_A(x),\nu_A(x)\rbrace $. It is clear that when $\mu_A(x) + \nu_A(x) = 1$, then $ A $ will be a fuzzy subset of $ X $. The class of all intuitionistic fuzzy subsets of $ X $ is denoted by $ IFS(X) $. (\cite{8},\cite{18})
    
In this section, we give some basic concepts of intuitionistic fuzzy sets and intuitionistic fuzzy modules.
For simplicity, we sometimes denote each intuitionistic fuzzy set $A = \lbrace x ,\mu_A(x),\nu_A(x)\rbrace$ by $ A=(\mu_A,\nu_A) $. 
\begin{definition}\label{d1.2} \cite{8}
Let $X$ be a non-empty set and $A=(\mu_A,\nu_A), B=(\mu_B,\nu_B)$ be intuitionistic fuzzy sets of $X$. Then
\begin{equation}
\label{eq1}
A \subseteq B \Leftrightarrow\mu_A(x) \leq \mu_B(x)\quad and\quad \nu_A(x) \geq \nu_B(x)\quad for\ all\: x\in X
\end{equation}
\begin{equation}
\label{eq2}
A \cup B = \lbrace(x, \mu_A(x)\vee \mu_B(x), \nu_A(x)\wedge \nu_B(x)) \vert x \in X\rbrace\\
\end{equation}
\begin{equation}
\label{eq3}
A \cap B = \lbrace(x, \mu_A(x)\wedge \mu_B(x), \nu_A(x)\vee \nu_B(x)) \vert x \in X\rbrace\\
\end{equation}
\end{definition}

\begin{definition}\cite{15}
Let R be a ring and $ A\in IFS(R)$. Then $ A $ is called an intuitionistic fuzzy ideal of R if for all $ x , y \in R$, the followings hold:

1. $\mu_A (x-y)\geq  \mu_A (x)\wedge \mu_A (y) \: ,\: \nu_A (x-y) \leq \nu_ (x) \vee \nu_A (y);$

2. $ \mu_A (xy) \geq \mu_A (x) \vee \mu_A (y) \: ,\: \nu_A (xy) \leq \nu_A (x) \wedge\nu_A (y).$
 
The class of intuitionistic fuzzy ideals of R is denoted by IFI(R).

\end{definition}
\begin{definition}\cite{3}
Let $ M $ be an R-module and $ A\in IFS(M) $. Then $ A $ is called an intuitionistic fuzzy submodule if

1. $\mu _{A}(\theta) = 1 \:  , \: \nu_A (\theta ) = 0$;

2. $\mu_A (x + y )\geq \mu_A (x ) \wedge \mu_A ( y ) \: , \: \nu_A(x + y )\leq \nu_A (x ) \vee\nu_A ( y)$ \quad for all $ x , y \in M $;

3. $ \mu_A (rx )\geq \mu_A ( x ) \: , \: \nu_A (rx ) \leq \nu_A (x)$ \quad for all $ x \in M$ and $ r\in R $.

We denote the class of all intuitionistic fuzzy submodule of an R-module $ M $, by $ IFM(M) $.

Notice that when $ R=M $, then $ A\in IFM(M) $ if and only if $ \mu _{A}(\theta) = 1 \:  , \: \nu_A (\theta ) = 0 $ and $ A\in IFI(R) $.
\end{definition}
 Now, for an intuitionistic fuzzy submodule, we define intuitionistic fuzzy radical.
 \begin{definition}

Let $M$ be an R-module and  $ A\in IFS(M)$. The intuitionistic fuzzy subset $\sqrt{A}=( \sqrt{\mu_A},\sqrt{\nu_A} )$ of $R$ is called the intuitionistic fuzzy radical of $A$ and defined as follows: 
\begin{equation}
\label{eq4}
 \sqrt{\mu_A}(r) =\bigvee\limits_{n\in\mathbb{N}} \bigwedge \limits_{m\in M}\mu_{A}(r^{n}.m)\quad and \quad \sqrt{\nu_A}(r)=\bigwedge\limits_{n\in\mathbb{N}}\bigvee\limits_{m\in M}\nu_A(r^{n}.m) 
\end{equation}
for all $ r\in R $, $ m\in M $, $ n\in \mathbb{N} $. 
\end{definition}
It is clear that when $ M=R $, then the intuitionistic fuzzy radical of an intuitionistic fuzzy ideal of $ R $ will be defined by:
$\sqrt{A}=(\sqrt{\mu_A},\sqrt{\nu_A})$ where $\sqrt{\mu_A}(r) =\bigvee\limits_{n\in\mathbb{N}} \mu_{A}(r^{n})$ and $\sqrt{\nu_A}(r) =\bigwedge\limits_{n\in\mathbb{N}} \nu_{A}(r^{n}) $.
Now we show that for any intuitionistic fuzzy submodule $ A $ of $ M $, $\sqrt{A}$ is an intuitionistic fuzzy ideal of R. 
\begin{proposition} \label{pr1.3}
Let $A=(\mu_A,\nu_A)$ be an intuitionistic fuzzy submodule of $M$ then $\sqrt{A}=(\sqrt{\mu_A},\sqrt{\nu_A})$ is an intuitionistic fuzzy ideal of R.
\end{proposition}
\begin{proof}
Let $r_1 , r_2\in R$ then

\begin{small}
$ \sqrt{\mu}_A(r_1+r_2)=\bigvee\limits_{n\in\mathbb{N}} \bigwedge \limits_{m\in M}\mu_{A}((r_1+r_2)^{n}.m)=
\lbrace\bigvee\limits_{n\in\mathbb{N}}\bigwedge\limits_{m\in M}\mu_A((r_1+r_2)^{2n}.m)\rbrace \vee\lbrace\bigvee\limits_{n\in\mathbb{N}}\bigwedge\limits_{m\in M}\mu_A((r_1+~r_2)^{2n+1}.m)\rbrace  $.
\end{small}

Now we have
\begin{equation*}
\begin{array}{ll}
&\bigwedge\limits_{m\in M}\mu_A((r_1+r_2)^{2n}.m)= \bigwedge\limits_{m\in M}\mu_A((\sum\limits_{i=0}^{2n}C_{i}^{2n}r_1^{2n-i}r_2^{i}).m)\geq\\
 &\bigwedge \limits_{m\in M}\lbrace\mu_A(r_1^{2n}.m)\wedge  \mu_A(r_1^{2n-1}.r_2.m)\wedge ... \wedge  \mu_A(r_1^{n}r_2^{n}.m)\wedge  \mu_A(r_1^{n-1}r_2^{n+1}.m)\wedge ... \wedge \mu_A(r_2^{2n}.m)\rbrace\\
& \geq \bigwedge \limits_{m\in M}\lbrace\mu_A(r_1^{2n}.m)\wedge \mu_A(r_1^{2n-1}.m)\wedge...\wedge\mu_A(r_1^{n}.m) ... \wedge  \mu_A(r_2^{n+1}.m)\wedge ... \wedge \mu_A(r_2^{2n}.m)\rbrace\\
& = \bigwedge \limits_{m\in M}\lbrace\mu_A(r_1^{n}.m)\wedge \mu_A(r_2^{n+1}.m)\rbrace\\
& \geq \bigwedge \limits_{m\in M}\lbrace\mu_A(r_1^{n}.m)\wedge \mu_A(r_2^{n}.m)\rbrace\\
& = (\bigwedge \limits_{m\in M}\mu_A(r_1^{n}.m))\wedge (\bigwedge \limits_{m\in M}\mu_A(r_2^{n}.m))
\end{array}
\end{equation*}
Also
\begin{equation*}
\begin{array}{ll}
&\bigwedge\limits_{m\in M}\mu_A((r_1+r_2)^{2n+1}.m)= \bigwedge\limits_{m\in M}\mu_A((\sum\limits_{i=0}^{2n+1}C_{i}^{2n+1}r_1^{2n+1-i}r_2^{i}).m)\\
&\geq \bigwedge \limits_{m\in M}\lbrace\mu_A(r_1^{2n+1}.m)\wedge  \mu_A(r_1^{2n}.r_2.m)\wedge ... \wedge  \mu_A(r_1^{n}r_2^{n+1}.m)\wedge  ... \wedge \mu_A(r_2^{2n+1}.m)\rbrace\\
& \geq \bigwedge \limits_{m\in M}\lbrace\mu_A(r_1^{2n+1}.m)\wedge \mu_A(r_1^{2n}.m)\wedge...\wedge\mu_A(r_1^{n+1}.m) ... \wedge  \mu_A(r_2^{n+1}.m)\wedge ... \wedge \mu_A(r_2^{2n+1}.m)\rbrace\\
& = \bigwedge \limits_{m\in M}\lbrace\mu_A(r_1^{n+1}.m)\wedge \mu_A(r_2^{n+1}.m)\rbrace\\
& \geq \bigwedge \limits_{m\in M}\lbrace\mu_A(r_1^{n}.m)\wedge \mu_A(r_2^{n}.m)\rbrace\\
& = (\bigwedge \limits_{m\in M}\mu_A(r_1^{n}.m))\wedge (\bigwedge \limits_{m\in M}\mu_A(r_2^{n}.m))
\end{array}
\end{equation*}
 Therefore
\begin{equation*}
\begin{array}{ll}
& \sqrt{\mu_A}(r_1+r_2) = \lbrace\bigvee\limits_{n\in\mathbb{N}}(\bigwedge\limits_{m\in M}\mu_A((r_1+r_2)^{2n}.m)\rbrace \vee\lbrace\bigvee\limits_{n\in\mathbb{N}}(\bigwedge\limits_{m\in M}\mu_A((r_1+r_2)^{2n+1}.m)\rbrace \\
 & \geq \bigvee\limits_{n\in\mathbb{N}}\lbrace(\bigwedge\limits_{m\in M}\mu_A(r_1^{n}.m))\wedge (\bigwedge\limits_{m\in M}\mu_A(r_2^{n}.m))\rbrace \vee \bigvee\limits_{n\in\mathbb{N}}\lbrace(\bigwedge\limits_{m\in M}\mu_A(r_1^{n}.m))\wedge (\bigwedge\limits_{m\in M}\mu_A(r_2^{n}.m))\rbrace\\
 & = \bigvee\limits_{n\in\mathbb{N}}\lbrace(\bigwedge\limits_{m\in M}\mu_A(r_1^{n}.m))\wedge (\bigwedge\limits_{m\in M}\mu_A(r_2^{n}.m))\rbrace\\
 & = (\bigvee\limits_{n\in\mathbb{N}}\bigwedge\limits_{m\in M}\mu_A(r_1^{n}.m))\wedge (\bigvee\limits_{n\in\mathbb{N}}\bigwedge\limits_{m\in M}\mu_A(r_2^{n}.m))\\
 & =\sqrt{\mu_A}(r_1)\wedge\sqrt{\mu_A}(r_2)
\end{array}
\end{equation*}
Similarly, we show that $\sqrt{\nu_A}(r_1+r_2)\leq\sqrt{\nu_A}(r_1)\vee\sqrt{\nu_A}(r_2).$
\begin{equation*}
\begin{array}{ll}
 \sqrt{\nu}_A(r_1+r_2) & =\bigwedge\limits_{n\in\mathbb{N}} \bigvee \limits_{m\in M}\nu_{A}((r_1+r_2)^{n}.m)=\bigwedge\limits_{n\in\mathbb{N}}\bigvee\limits_{m\in M}\nu_A((\sum\limits_{i=0}^{n}C_{i}^{n}r_1^{n-i}r_2^{i}).m)\\
& \leq\bigwedge\limits_{n\in\mathbb{N}}\bigvee \limits_{m\in M}\lbrace\nu_A(r_1^{n}.m)\vee\nu_A(r_1^{n-1}r_{2}.m)\vee ... \vee  \nu_A(r_1r_2^{n-1}.m)\vee \nu_A(r_2^{n}.m)\rbrace\\
& \leq \bigwedge \limits_{n\in\mathbb{N}}\bigvee \limits_{m\in M}\lbrace\nu_A(r_1^{n}.m)\vee \nu_A(r_1^{n-1}m)\vee...\vee\nu_A(r_2^{n-1}.m) \vee \nu_A(r_2^{n}.m)\rbrace\\
&(since\ A\ is\ an\ intuitionistic\ fuzzy\ submodule\ of\ M)\\
& = \bigwedge \limits_{n\in\mathbb{N}}\bigvee \limits_{m\in M}\lbrace\nu_A(r_1^{n}.m)\vee \nu_A(r_2^{n}.m)\rbrace\\
& = (\bigwedge \limits_{n\in\mathbb{N}}\bigvee \limits_{m\in M}\nu_A(r_1^{n}.m))\vee (\bigwedge \limits_{n\in\mathbb{N}}\bigvee \limits_{m\in M}\nu_A(r_2^{n}.m))\\
& =\sqrt{\nu_A}(r_1)\vee\sqrt{\nu_A}(r_2)
\end{array}
\end{equation*}

And
\begin{small}
$\sqrt{\mu_A}(r_1r_2)= \lbrace\bigvee\limits_{n\in\mathbb{N}}(\bigwedge\limits_{m\in M}\mu_A((r_1r_2)^{n}.m)\rbrace = \lbrace\bigvee\limits_{n\in\mathbb{N}}(\bigwedge\limits_{m\in M}\mu_A((r_1^{n}r_2^{n}.m)\rbrace \geq$
 $\lbrace\bigvee\limits_{n\in\mathbb{N}}(\bigwedge\limits_{m\in M}\mu_A((r_1^{n}.m)\rbrace= \\
  \sqrt{\mu_A}(r_1)$.
\end{small}
Similarly, $\sqrt{\mu_A}(r_1r_2)\geq\sqrt{\mu_A}(r_2).$ 
Hence $\sqrt{\mu_A}(r_1r_2)\geq\sqrt{\mu_A}(r_1)\vee\sqrt{\mu_A}(r_2).$

Also,$\sqrt{\nu_A}(r_1r_2)= \sqrt{\nu_A}(r_1)$ and 
$\sqrt{\nu_A}(r_1r_2)\leq\sqrt{\nu_A}(r_2)$, hence $\sqrt{\nu_A}(r_1r_2)\leq\sqrt{\nu_A}(r_1)\wedge\sqrt{\nu_A}(r_2).$
Next for $ r\in R $
 
$\sqrt{\mu_A}(-r)=\bigvee\limits_{n\in\mathbb{N}}(\bigwedge\limits_{m\in M}\mu_A((-r)^{n}.m) =\bigvee\limits_{n\in\mathbb{N}}(\bigwedge\limits_{m\in M}\mu_A(r^{n}.m) =\sqrt{\mu_A}(r)$;

$\sqrt{\nu_A}(-r)=\bigwedge\limits_{n\in\mathbb{N}}(\bigvee\limits_{m\in M}\nu_A((-r)^{n}.m) =\bigwedge\limits_{n\in\mathbb{N}}(\bigvee\limits_{m\in M}\nu_A(r^{n}m))=\sqrt{\nu_A}(r).$

Therefore $ \sqrt{A} $ is an intuitionistic fuzzy ideal of $R$. 
\end{proof}
In the following proposition we will give some properties of intuitionistic fuzzy radicals.
\begin{proposition}
Let $ A $, $ B \in IFM(M)$ then\\
1. $\sqrt{\sqrt{A}}=\sqrt{A}$\\
2. $A\subseteq B \Rightarrow \sqrt{A}\subseteq \sqrt{B}$\\
3. $ \sqrt{A\cap B}=\sqrt{A}\cap \sqrt{B} $\\
4. $ \sqrt{\sqrt{A}+\sqrt{B}}\subseteq \sqrt{A+B} $\\
\end{proposition}
\begin{proof}
1. Since$ \sqrt{\sqrt{A}}=(\sqrt{\sqrt{\mu}_A}, \sqrt{\sqrt{\nu}_A} )$, so by Proposition \ref{pr1.3} $ \sqrt{\sqrt{A}}  $ is an intuitionstic fuzzy ideal of $ R $ then
$\sqrt{\sqrt{\mu_A}}(r)=\bigvee\limits_{n\in\mathbb{N}}\sqrt{\mu_A}(r^n)
=\bigvee\limits_{n\in\mathbb{N}}(\bigvee\limits_{n'\in\mathbb{N}}\bigwedge\limits_{m\in M}\mu_A ((r^{n})^{n'}.m)=\bigvee\limits_{n\in\mathbb{N}}\bigvee\limits_{n'\in\mathbb{N}}(\bigwedge\limits_{m\in M}\mu_A(r^{nn'}.m))$
$=\bigvee\limits_{n''\in\mathbb{N}}\bigwedge\limits_{m\in M}\mu_A(r^{n''}.m)=\sqrt{\mu_A}(r).$   
It is clear that $\sqrt{\sqrt{\nu_A}}(r)=\sqrt{\nu_A}(r)$, since the proof is the dual of the proof of the first part.
Hence $\sqrt{\sqrt{A}}=(\sqrt{\mu_A},\sqrt{\nu_A})=\sqrt{A}$.\\
3. By \ref{eq3}\ ,\ $A \cap B = \lbrace(m, \mu_A(m)\wedge \mu_B(m), \nu_A(m)\vee \nu_B(m)) : m \in M\rbrace$, hence
\begin{equation*}
\begin{array}{ll}
\sqrt{\mu_A \wedge \mu_B}(r) &= \bigvee\limits_{n\in\mathbb{N}}\bigwedge\limits_{m\in M}(\mu_A \wedge \mu_B)({r^n}.m)\\
 &=\bigvee\limits_{n\in\mathbb{N}}\bigwedge\limits_{m\in M}(\mu_A({r^n}.m) \wedge \mu_B({r^n}.m))\\
 &=(\bigvee\limits_{n\in\mathbb{N}}\bigwedge\limits_{m\in M}\mu_A({r^n}.m))\wedge (\bigvee\limits_{n\in\mathbb{N}}\bigwedge\limits_{m\in M}\mu_B({r^n}.m))\\
 &=(\sqrt{\mu_A}(r)\wedge \sqrt{\mu_A}(r)).
 \end{array}
 \end{equation*}
Similarly,
\begin{equation*}
\begin{array}{ll}
\sqrt{\nu_A \vee \nu_B}(r) &= \bigwedge\limits_{n\in\mathbb{N}}\bigvee\limits_{m\in M}(\nu_A \vee \nu_B)({r^n}.m)\\
 &=\bigwedge\limits_{n\in\mathbb{N}}\bigvee\limits_{m\in M}(\nu_A({r^n}.m) \vee \nu_B({r^n}.m))\\
 &=(\bigwedge\limits_{n\in\mathbb{N}}\bigvee\limits_{m\in M}\nu_A({r^n}.m))\vee (\bigwedge\limits_{n\in\mathbb{N}}\bigvee\limits_{m\in M}\nu_B({r^n}.m))\\
 &=(\sqrt{\nu_A}(r)\vee \sqrt{\nu_A}(r)). 
 \end{array}
\end{equation*}
Therefore $ \sqrt{A\cap B}=(r,\sqrt{\mu_A}(r)\wedge \sqrt{\mu_B}(r), \sqrt{\nu_A}(r)\vee \sqrt{\nu_B}(r))=\sqrt{A}\cap \sqrt{B}  $.

4. $ \mu_A \subseteq \mu_A + \mu_B $ \ because 
$ (\mu_A + \mu_B)(x)=\vee\lbrace \mu_A(y) \wedge \mu_B(z)\ \vert \ y+z=x\rbrace \geq \lbrace \mu_A(x) \wedge \mu_B(\theta)\ ;\ \mu_B(\theta)=1\rbrace=\mu_A(x)$.
 And similarly, $ \mu_B \subseteq \mu_A + \mu_B. $ 
 
 In the same way, we have $ \nu_A\supseteq \nu_A + \nu_B $ because
  $(\nu_A + \nu_B)(x)=\wedge\lbrace \nu_A(y) \vee \nu_B(z)\ \vert\ y+z=x\rbrace \leq \lbrace \nu_A(x) \vee \nu_B(\theta)\ ;\ \nu_B(\theta)=0\rbrace=\nu_A(x)$. Also $ \nu_B\supseteq \nu_A+\nu_B $.
 
Now we have  $ \sqrt{\mu_A} \subseteq \sqrt{\mu_A + \mu_B} $ \ , \ $ \sqrt{\mu_B} \subseteq \sqrt{\mu_A + \mu_B}$ , hence $ \sqrt{\mu_A}+\sqrt{\mu_B}\subseteq \sqrt{\mu_A + \mu_B}\ \Longrightarrow\ \sqrt{\sqrt{\mu_A}+\sqrt{\mu_B}}\subseteq \sqrt{\sqrt{\mu_A + \mu_B}}=\sqrt{\mu_A + \mu_B}\ . $\\  
And also\  $\sqrt{\nu_A} \supseteq \sqrt{\nu_A + \nu_B} $ \ , \ $ \sqrt{\nu_B} \supseteq \sqrt{\nu_A + \nu_B} $\ hence
$ \sqrt{\nu_A}+\sqrt{\nu_B}\supseteq \sqrt{\nu_A + \nu_B}\ \Longrightarrow\ \sqrt{\sqrt{\nu_A}+\sqrt{\nu_B}}\supseteq \sqrt{\sqrt{\nu_A + \nu_B}}=\sqrt{\nu_A + \nu_B}$. Therefore $ \sqrt{\sqrt{A}+ \sqrt{B}}\subseteq \sqrt{A+B} $. 
\end{proof}
\begin{definition} \cite{17}
Let $ A $ be a submodule of $ M $. Then intuitionistic fuzzy characteristic function $\chi_{A}=(\mu_{\chi_A} ,\nu_{\chi_A} )$ is defined as follows\\
\begin{equation}
\label{eq5}
\mu_{\chi_A}(x)=\begin{cases}
1 & \text{if\ $x\in A$}\\
0 & \text{if\ $x\notin A$}
\end{cases}\quad  ,\qquad
\nu_{\chi_A}(x)=\begin{cases}
0 & \text{if \ $x\in A$}\\
1 & \text{if\ $x\notin A$}
\end{cases} 
\end{equation}
\end{definition}
\begin{proposition}\label{pr2.3} 
Let $ A $ be a submodule of $ M $, then $ \sqrt{\chi_A}=\chi_{\sqrt{A}} $ where $ \chi_A $ is intuitionistic fuzzy characteristic function of $ A $.
\end{proposition}
\begin{proof}
Let $ r\in R $, then $ \sqrt{\chi_A}(r)=(\bigvee\limits_{n\in \mathbb{N}}\bigwedge\limits_{m\in M}\mu_{\chi_A} (r^{n}.m) , \bigwedge\limits_{n\in \mathbb{N}}\bigvee\limits_{m\in M}\nu_{\chi_A} (r^{n}.m)).$ 
Now we consider two cases:\\
$ \textit{case 1 :} $ $ r^{n}.m\in A $ then $ \mu_{\chi_A} (r^{n}.m)=1 , \nu_{\chi_A} (r^{n}.m)=0 $ and hence 
$\sqrt{\chi_A}(r)=(1,0) $. By $ r^{n}.m\in A  $ it follows that $ r\in \sqrt{A} $ and then $ \chi_{\sqrt{A}}(r)=(1, 0) $ this mean that $ \sqrt{\chi_A}=\chi_{\sqrt{A}} $ \\
$ \textit{case 2 :} $ $ r^{n}.m\notin A $ then $ \mu_{\chi_A} (r^{n}.m)=0 , \nu_{\chi_A} (r^{n}.m)=1 $ and $\sqrt{\chi_A}(r)=(0,1) $. Now since  $ r^{n}.m\notin A $ for all $ n\in\mathbb{N} $ then $ r\notin \sqrt{A} $ hence $ \chi_{\sqrt{A}}(r)=(0,1) $ and therefore for this case  $ \sqrt{\chi_A}=\chi_{\sqrt{A}}.$
\end{proof}
\begin{definition}\cite{9}
If $ A\in IFS(R) $, then $ A $ is said to have the $sup\ property$ if for every $\emptyset\neq Y\subseteq R $, there is a $ y_0 \in Y $ such that $sup_{y\in Y} \mu_A(y)=\mu_A(y_0) $ , $ inf_{y\in Y} \nu_A(y)=\nu_A(y_0) $.  
\end{definition}
\begin{definition}Suppose that $ A\in IFS(X)$. Then the set $ X ^{(\alpha,\beta)}_{A} = \lbrace x\mid x \in X \: $such that  $\: \mu_A (x )\geq \alpha{ }\   and\   \nu_A (x ) \leq \beta \rbrace$, where $\alpha,\beta \in [0,1]$ with $\alpha+\beta\leq1$, is called $(\alpha,\beta)$- cut set (crisp set ).
\end{definition} 
\begin{proposition}\label{pr3.3} 
Let $A \in IFM(M)$ and suppose that $A $ has sup property, then $ (\sqrt{M_A})^{(\alpha,\beta)}=\sqrt{M_{A}^{(\alpha,\beta)}} $, where 
$ M_{A}^{(\alpha,\beta)}= \lbrace x\in M\ \vert\  \mu_A(x)\geq \alpha \ , \ \nu_A(x)\leq \beta \ ; \ \alpha+\beta\leq 1 \ ; \ (\alpha , \beta)\in Im \mu_A \times Im \nu_A \rbrace $.
\end{proposition}
\begin{proof} Let $ A $ be an intuitionistic fuzzy submodule of $ M $ with sup property, then
\begin{equation*}
\begin{array}{ll}
 (\sqrt{M_A})^{(\alpha,\beta)}&=\lbrace r\in R \ \vert \ \sqrt{\mu_A}(r)\geq \alpha \ , \ \sqrt{\nu_A}(r)\leq \beta \rbrace\\
&=\lbrace r\in R \ \vert \ \bigvee\limits_{n\in \mathbb{N}}\bigwedge\limits_{m\in M}\mu_A (r^{n}.m)\geq \alpha \ , \ \bigwedge\limits_{n\in \mathbb{N}}\bigvee\limits_{m\in M}\nu_A (r^{n}.m)\leq \beta \rbrace\\
&=\lbrace  r\in R \ \vert \   \bigwedge\limits_{m\in M}\mu_A (r^{n_1}.m)\geq\alpha\ , \ 
\bigvee\limits_{m\in M}\nu_A (r^{n_2}.m)\leq\beta \ for\ some\ n_1 , n_2 \in \mathbb{N} \rbrace\\
& ($Since$\ A\ $has$\ sup\ property ) \\
&= \lbrace r\in R \ \vert \ \mu_A (r^{n_1}r^{n_2}.m)\geq\alpha \ , \ \nu_A (r^{n_2}r^{n_1}.m)\leq\beta \ for\ all\ m\in M\rbrace\\ 
 &($Since$\ \mu_A(r^{n_1}r^{n_2}.m)\geq \mu_A(r^{n_1}.m)\geq \alpha , \nu_A(r^{n_1}r^{n_2}.m)\leq \nu_A(r^{n_2}.m)\leq \beta  ) \\
&=\lbrace r\in R \ \vert \ n'' =n_1 + n_2\ ; \ \mu_A(r^{n''}.m)\geq \alpha \ , \ \nu_A(r^{n''}.m)\leq \beta , \quad for\ all\ m\in M \rbrace \\
&=\lbrace r\in R \ \vert \  r^{n''}.m \in {M_A}^{(\alpha , \beta)} , \quad  for\ all\ m\in M \, , \, for\ some\ n''\in \mathbb{N} \rbrace\\
&=\lbrace r\in R \ \vert \  \ r^{n''}.M \subseteq {M_A}^{(\alpha , \beta)}  for\ some\ n''\in \mathbb{N} \rbrace\\
&=\sqrt{M_{A}^{(\alpha,\beta)}} 
\end{array}
\end{equation*}
\end{proof}
The above proposition is also true for strict level cuts, where the $ sup \ property$ is not required.
\section{\textbf{Intuitionistic Fuzzy Primary Submodules}}
In this section, we will give some characterization of an intuitionistic fuzzy primary submodule. we begin with a definition.
\begin{definition}
Let $A = (\mu_A,\nu_A)$ be an intuitionistic fuzzy submodule of $M$. Then the intuitionistic fuzzy subset $ \bar{A}= (\mu_{\bar{A}}, \nu_{\bar{A}}) $, is defined 
\begin{equation}
\mu_{\bar{A}}(r)=\bigwedge\limits_{m\in M} \mu_A(r.m) \ , \  \nu_{\bar{A}}(r)=\bigvee\limits_{m\in M} \nu_A(rm)
 \end{equation}
 where $ \bar{A}=(A:M) $. 
\end{definition}
\begin{proposition}
Let $ A $ be an intuitionistic fuzzy submodule of $ M $. Then $\bar{A}$ is an intuitionistic fuzzy ideal of $ R $.
\end{proposition}
\begin{proof}
Let $ r_1 , r_2 \in R $ then

$
\mu_{\bar{A}}(r_{1} - r_{2})= \bigwedge\limits_{m\in M} \mu_{A}((r_{1} - r_{2}).m)
=\bigwedge\limits_{m\in M} \mu_A(r_{1}.m - r_{2}.m)
\geq \bigwedge\limits_{m\in M}(\mu_A(r_{1}.m) \wedge \mu_A(r_{2}.m))
=(\bigwedge\limits_{m\in M}\mu_A(r_{1}.m))\wedge (\bigwedge\limits_{m\in M}\mu_A(r_{2}.m))
=\mu_{\bar{A}}(r_{1})\wedge \mu_{\bar{A}}(r_{2}).
$
It is clear that $ \nu_{\bar{A}}(r_{1} - r_{2})\leq\nu_{\bar{A}}(r_{1})\vee \mu_{\bar{A}}(r_{2})$. Since the proof is the dual of the proof of the first part, hence we got that
$\mu_{\bar{A}}(r_{1} - r_{2})\geq\mu_{\bar{A}}(r_{1})\wedge \mu_{\bar{A}}(r_{2})$\ , \ $ \nu_{\bar{A}}(r_{1} - r_{2})\leq\nu_{\bar{A}}(r_{1})\vee \mu_{\bar{A}}(r_{2})$. Now suppose that $ r_1 , r_2 \in R $. Then  we have  
$ \mu_{\bar{A}}(r_{1}r_{2})=\bigwedge\limits_{m\in M} \mu_{A}(r_{1}r_{2}.m)\geq \bigwedge\limits_{m\in M} \mu_{A}(r_{1}.m)=\mu_{\bar{A}}(r_{1})$, similarly $ \nu_{\bar{A}}(r_{1}r_{2})\geq \nu_{\bar{A}}(r_{2}) $. Now $ \mu_{\bar{A}}(r_{1}r_{2})\geq \mu_{\bar{A}}(r_{1})\vee \mu_{\bar{A}}(r_{2}) $\ and 
also \ $ \nu_{\bar{A}}(r_{1}r_{2})\leq \nu_{\bar{A}}(r_{1})\wedge \nu_{\bar{A}}(r_{2}) $. Therefore $ \bar{A} $ is an intuitionistic fuzzy ideal of $ R $.
\end{proof}
\begin{remark}\label{re1.3}
If $ A \in IFM(M) $ then its radical is equal to the radical of  $ \bar{A} $, because:
\begin{equation*}
\begin{array}{ll}
 \sqrt{\bar{A}}&=\lbrace r,\sqrt{\mu_{\bar{A}}}(r),\sqrt{\nu_{\bar{A}}}(r)\rbrace\\
 &=\lbrace r,\bigvee\limits_{n\in \mathbb{N}}\mu_{\bar{A}}(r^n),\bigwedge\limits_{n\in \mathbb{N}}\nu_{\bar{A}}(r^n)\rbrace\qquad $(Since\ $\bar{A}$\  is\ an\  intuitionistic\ fuzzy\ ideal\  of\  $R$ )$\\
 &=\lbrace r,\bigvee\limits_{n\in \mathbb{N}}\bigwedge\limits_{m\in M}\mu_A({r^n}.m), \bigwedge\limits_{n\in \mathbb{N}}\bigvee\limits_{m\in M}\nu_A({r^n}.m)\rbrace\\
 &=\lbrace r,\sqrt{\mu_{A}}(r),\sqrt{\nu_{A}}(r)\rbrace =\sqrt{A}. 
\end{array}
\end{equation*}
\end{remark}
\begin{proposition}
Let $M$ be an  R-module and $ A\in IFM(M) $. Then $  M ^{(\alpha,\beta)}_{\bar{A}}= \overline{M_{A}^{(\alpha,\beta)}}  $, where
 $ M ^{(\alpha,\beta)}_{A} = \lbrace x \in M : \mu_A (x )\geq \alpha{ }\   and\   \nu_A (x ) \leq \beta \rbrace$.
\end{proposition}
\begin{proof}
\begin{equation*}
\begin{array}{ll}
 \overline{M_{A}^{(\alpha,\beta)}} &=(M ^{(\alpha,\beta)}_{A}:M)\\
 &= Ann\frac{M}{M ^{(\alpha,\beta)}_{A}}\\
 &=\lbrace r\in R\  \vert \ r.(m+M ^{(\alpha,\beta)}_{A})= M ^{(\alpha,\beta)}_{A} \quad for\ all\ m\in M \rbrace\\
 &= \lbrace r\in R\  \vert\ r.m+M ^{(\alpha,\beta)}_{A}=M ^{(\alpha,\beta)}_{A} \quad for\ all\ m\in M \rbrace\\
 & = \lbrace r\in R\ \vert\  r.m\in M ^{(\alpha,\beta)}_{A} \quad for\ all\ m\in M \rbrace\\
 & =\lbrace r\in R\ \vert\ \mu_{A}(r.m)\geq \alpha  , \nu_{A}(r.m)\leq \beta  \quad for\ all\ m\in M \rbrace\\
 &= \lbrace r\in R\ \vert\ \bigwedge\limits_{m\in M}\mu_{A}(r.m)\geq \alpha  , \bigvee\limits_{m\in M}\nu_{A}(r.m)\leq \beta \rbrace\\
 & = \lbrace r\in R\ \vert\  \mu_{\bar{A}}(r)\geq \alpha  , \nu_{\bar{A}}(r)\leq \beta \rbrace\\
 & = M ^{(\alpha,\beta)}_{\bar{A}}    
\end{array}
\end{equation*}
\end{proof}
\begin{definition} \cite{18}
Let M and N be modules over the same ring R and let $f$ be a mapping from M to N. Let $ A = (\mu_A, \nu_A)$ be an intuitionistic fuzzy submodule of $ M$ and $B = (\mu_B,\nu_B)$ be an intuitionistic fuzzy submodule of $N$. Then the image of intuitionistic fuzzy submodule $ A $ of $M$, $f(A) = (f(\mu_A),\nu_{A}(f)) \in N$ can be defined for all $y \in N$ as follows:
\begin{footnotesize}
\begin{equation*}
f(\mu_A)(y)=\begin{cases}
\vee\lbrace \mu_{A}(x)\ \vert \ x\in R \ , f(x)=y\rbrace
& \text{if $f^{-1}(y)\neq \emptyset $} \\
0 & \text{otherwise}
\end{cases} \\,
\nu_{A}(f)=\begin{cases}
\wedge\lbrace \nu_{A}(x)\ \vert \ x\in R \ , f(x)=y\rbrace
& \text{if $f^{-1}(y)\neq \emptyset $} \\
1 & \text{otherwise}
\end{cases} 
\end{equation*}
\end{footnotesize}
and inverse image of intuitionistic fuzzy submodule $ B $ of $N$ can be defined as $x \in M, f^{-1}(B) = B(f(x))$.
\end{definition}
\begin{theorem}\label{th1.3} 

Let $ f:M\longrightarrow M' $ be an epimorphism of  R-modules. If $ A \in IFM(M) $ then $ \bar{A}\subseteq \overline{f(A)} $.

\end{theorem}
\begin{proof}
First, we show that $ f(A) $ be an intuitionistic fuzzy submodule of $ M' $. Let $ u,v\in R $, then there exists $ x,y \in R $ such that $ f(x)=u,f(y)=v $  hence $ f(x-y)=u-v $. Now we have:

1)$ f(\mu_A)(u-v)=\bigvee\limits_{z\in f^{-1}(u-v)}\mu_A(z)
\geq \bigvee\limits_{x-y\in f^{-1}(u-v)}\mu_A(x-y)
\geq \bigvee\limits_{\substack{x\in f^{-1}(u) \\ y\in f^{-1}(v)}}(\mu_A(x)\wedge \mu_A(y))\\
=(\bigvee\limits_{x\in f^{-1}(u)}\mu_A(x))\wedge (\bigvee\limits_{y\in f^{-1}(v)}\mu_A(y))
=f(\mu_A)(u)\wedge f(\mu_A)(v)$.\\ 

2)$f(\mu_A)(r.m)=\bigvee\limits_{m'\in f^{-1}(r.m)}\mu_A(m')
\geq \bigvee\limits_{m'\in f^{-1}(m)}\mu_A(r.m')
\geq \bigvee\limits_{m'\in f^{-1}(m)}\mu_A(m')
=f(\mu_A)(m)$.\\

Similarly, $\nu_A(f)(u-v)=\nu_A(f)(u)\wedge \nu_A(f)(v)\ and \; \nu_A(f)(r.m)=\nu_A(f)(m) $.\\

3)$ f(\mu_A)(0_{M'})=\bigvee\limits_{m\in f^{-1}(0_{M'})}\mu_A(m)=\mu_A(0_M)$\,,\,$\nu_A(f)(0_{M'})=\bigwedge\limits_{m\in f^{-1}(0_{M'})}\nu_A(m)=\nu_A(0_M) $.\\

Thus we prove that $ f(A)\in IFM(M') $.
Let $ r\in R$, then\\ 
\begin{footnotesize}
$\overline{f(\mu_{A})}(r)= (f(\mu_A):f(M))(r)=\bigwedge\limits_{m'\in f(M)}f(\mu_A)(r.m')=\bigwedge\limits_{m'\in f(M)}\ \bigvee\limits_{m\in f^{-1}(r.m')}\mu_A(m)\geq \bigwedge\limits_{m'\in f(M)}\ \bigvee\limits_{r.m\in f^{-1}(r.m')}\mu_A(r.m)\\
(Sice\  m\in f^{-1}(m')\  then\;  r.m\in f^{-1}(r.m'))\geq \bigwedge\limits_{\substack{m'\in f(M) \\ m\in f^{-1}(m')}}\mu_{A}(r.m)=\bigwedge\limits_{\substack{ m'\in f(M) \\ m\in \cup f^{-1}(m')}}\mu_{A}(r.m)=\bigwedge\limits_{\substack{ m'\in f(M) \\ m\in \cup f^{-1}(m')}}\mu_{A}(r.m)=\bigwedge\limits_{m\in M}\mu_{A}(r.m)=\mu_{\bar{A}}(r)$
\end{footnotesize}

 It is clear that $\overline{\nu_{A}(f)}(r)=\nu_{\bar{A}}(r) $.
Therefore $ \bar{A}\subseteq \overline{f(A)}=(\overline{f(\mu_{A})},\overline{\nu_{A}(f)}) $
\end{proof}
\begin{theorem}\label{th2.3} 
Let $ f:M\longrightarrow M' $ be a homomorphism of R-modules. If $ B\in IFM(M') $, then $ \overline{f^{-1}(B)}\supseteq \bar{B} $. Equality holds if\ $ f $\ is an epimorphism.
\end{theorem}
\begin{proof}
Let$ r\in R$, then

$\overline{f^{-1}(\mu_{B})}(r)=(f^{-1}(\mu_B):f^{-1}(M'))(r)
= \bigwedge\limits_{m\in f^{-1}(M') }f^{-1}(\mu_{B})(r.m)
=\bigwedge\limits_{m\in f^{-1}(M') }\mu_{B}(f(r.m))\\
=\bigwedge\limits_{m\in f^{-1}(M') }\mu_{B}(rf(m))
=\bigwedge\limits_{f(m)\in M'}\mu_{B}(r.f(m))
\geq \bigwedge\limits_{m'\in M'}\mu_{B}(r.m'))
=\mu_{\bar{B}}(r)$.
 Similar to the previous part, we can show that $\overline{f^{-1}(\nu_B)}(r)\leq \nu_{\bar{B}}(r)$.
Therefore  $ \overline{f^{-1}(B)}\supseteq \bar{B}$.
Furthermore if $ f $ is an epeimorphism then $ {f^{-1}}(\bar{B})= \bar{B}\ $.
\end{proof}
\begin{proposition}
Let $ A,B \in IFM(M)$. Then the following hold

1. $\bar{\bar{A}}=\bar{A}$

2. $ A\subseteq B \Longrightarrow \bar{A}\subseteq \bar{B} $

3. $ \overline{A\cap B}= \bar{A}\cap \bar{B} $

4. $ \overline{A+B}\supseteq \bar{A}+\bar{B} $

\end{proposition}
\begin{proof}
Let $ r\in R$, then

1.Since we can consider $ \bar{A} $ as $ R-module $, then
$ \mu_{\bar{\bar{A}}}(r)=\bigwedge\limits_{r'\in R} \mu_{\bar{A}}(r.r')=\mu_{\bar{A}}(r) $\ ,\ and also \ 
$ \nu_{\bar{\bar{A}}}(r)=\bigvee\limits_{r'\in R} \nu_{\bar{A}}(r.r')=\nu_{\bar{A}}(r)$ (Since\ $ 1\in R ) $.
Hence $ \bar{\bar{A}}= (\mu_{\bar{\bar{A}}} , \nu_{\bar{\bar{A}}})=(\mu_{\bar{A}} , \mu_{\bar{A}})= \bar{A} $

2.Since $ A\subseteq B $ then $\mu_A \subseteq \mu_B $ and $\nu_A \supseteq \nu_B $, hence
$ \mu_{\bar{A}}(r)= \bigwedge\limits_{m\in M} \mu_A(r.m)\leq \bigwedge\limits_{m\in M} \mu_B(r.m)=\mu_{\bar{B}}(r)\Longrightarrow \mu_{\bar{A}}\subseteq \mu_{\bar{B}} $, and
$ \nu_{\bar{A}}(r)= \bigvee\limits_{m\in M} \nu_A(r.m)\geq \bigvee\limits_{m\in M} \nu_B(r.m)=\nu_{\bar{B}}\Longrightarrow \nu_{\bar{A}}\supseteq \nu_{\bar{B}}(r)$. Therefore $ \bar{A}\subseteq \bar{B}$.

3.We know that $ A\cap B=(\mu_{A \cap B} , \nu_{A \cup B}) $, so
 $\mu_{\overline{A \cap B}}(r)=\bigwedge\limits_{m\in M}(\mu_{A \cap B})(rm)
= \bigwedge\limits_{m\in M}(\mu_{A}(rm) \wedge \mu_{B}(rm))
=(\bigwedge\limits_{m\in M}\mu_{A}(rm))\wedge(\bigwedge\limits_{m\in M}\mu_{B}(rm))
=\mu_{\bar{A}}(r)\wedge \mu_{\bar{B}}(r)
= (\mu_{\bar{A}}\cap \mu_{\bar{B}})(r)$.
Similarly, $\nu_{\overline{A \cup B}}(r)= (\nu_{\bar{A}}\cup \nu_{\bar{B}})(r)$ and so we get $ \overline{A \cap B}= \bar{A}\cap \bar{B}$.

4.It is clear that $ \mu_A \subseteq \mu_A + \mu_B $ , $ \mu_B \subseteq \mu_A + \mu_B $\ and \ $ \nu_A \supseteq \nu_A + \nu_B $ , $ \nu_B \supseteq \nu_A + \nu_B$. So 
$ \mu_{\bar{A}} \subseteq \mu_{\overline{A +B}}\ ,\  \mu_{\bar{B}} \subseteq \mu_{\overline{A +B}}$. Hence
 $\mu_{\bar{A}}+\mu_{\bar{B}}\subseteq \mu_{\overline{A +B}}  $\ and \ $ \nu_{\bar{A}} \supseteq \nu_{\overline{A +B}}\ ,\  \nu_{\bar{B}} \supseteq \nu_{\overline{A +B}}$ and so
 $\nu_{\bar{A}}+\nu_{\bar{B}}\supseteq \nu_{\overline{A +B}} $.
Therefore $ \overline{A +B}\supseteq \bar{A}+\bar{B} $
\end{proof}
In this part of the paper we define intuitionistic fuzzy primary submodule.
\begin{definition}
Let $ A $ be an intuitionistic fuzzy submodule of R-module $ M $, then $ A $ is called intuitionistic fuzzy primary submodule of $ M $, if for $ r\in R , m\in M $
\begin{equation}
\mu_A(r.m)=\mu_A(m)\ ,\ \nu_A(r.m)=\nu_A(m)\quad or \quad \mu_A(r.m)\leq \sqrt{\mu_A}(r)\  ,\ \nu_A(r.m)\geq \sqrt{\nu_A}(r)
\end{equation}
\end{definition}
\begin{example}
consider $ \mathbb{Z}_4 $ as a $ \mathbb{Z}-module $. Define $ A=(\mu_A , \nu_A) $ by

$
\mu_{A}(z)=\begin{cases}
1 & \text{if\ $z= \bar{0}$}\\
0 & \text{if\ $z= \bar{1}$}\\
\frac{1}{2} & \text{if\ $ z=\bar{2}$}\\
0 & \text{if\ $z= \bar{3}$}
\end{cases}\quad  ,\qquad
\nu_{A}(z)=\begin{cases}
0 & \text{if\ $z= \bar{0}$}\\
1 & \text{if\ $z= \bar{1}$}\\
\frac{1}{2} & \text{if\ $ z=\bar{2}$}\\
1 & \text{if\ $z= \bar{3}$}
\end{cases}
$

It is clear that $ A $ is an intuitionstic fuzzy submodule of $ \mathbb{Z}_4$. Since 
$\sqrt{\mu_A}(z)=\bigvee\limits_{n\in \mathbb{N}}\bigwedge\limits_{m\in \mathbb{Z}_4}\mu_A (z^{n}m)$ and $\sqrt{\nu_A}(z)=\bigwedge\limits_{n\in \mathbb{N}}\bigvee\limits_{m\in \mathbb{Z}_4}\nu_A (z^{n}m)$, then 

$\sqrt{\mu_{A}}(z)=\begin{cases}
1 & \text{if\ $z= 0$}\\
1 & \text{if\ $z\in <2>$}\\
0 & \text{if\ $z\notin <2>$}
\end{cases}\quad  ,\qquad
\sqrt{\nu_{A}}(z)=\begin{cases}
0 & \text{if\ $z= 0$}\\
0 & \text{if\ $z\in <2>$}\\
1 & \text{if\ $z\notin <2>$}
\end{cases}
$,

for all $ z\in \mathbb{Z}$. Thus $ A=(\mu_A , \nu_A)$ is an intuitionstic fuzzy primary submodule of $ \mathbb{Z}$.
\end{example}
Now we would like to investigate the relationship between intuitionistic fuzzy primary and primary submodules of a module.
\begin{proposition}
Let $ A $ be a submodule of $ M $. Then its intuitionistic fuzzy characteristic function $ \chi_A $ is an intuitionistic fuzzy primary submodule if and only if $ A $ is a primary submodule of $ M $.\
\end{proposition}
\begin{proof}
Let $ A $ be a primary submodule of $ M $ and $ r\in R, m\in M $. By \ref{eq5} we know that $ \chi_A =(\mu_{\chi_A} , \nu_{\chi_A})$ where 
\begin{center}
$
\mu_{\chi_A}(x)=\begin{cases}
1 & \text{if\ $x\in A$}\\
0 & \text{if\ $x\notin A$}
\end{cases}\quad  ,\qquad
\nu_{\chi_A}(x)=\begin{cases}
0 & \text{if \ $x\in A$}\\
1 & \text{if\ $x\notin A$}
\end{cases}
$
\end{center} 
If $ \chi_A(rm)=(1,0) $ then $ rm\in A $. Since $ A $ is a primary submodule, then $ m\in A  $ or $ r\in \sqrt{A} $. Now we consider two cases:

$ \textit{Case 1}{}: $ $ m\in A $ then $ \chi_A(m)=(1,0) $ and hence $ \chi_A(rm)=\chi_A(m) $. 

$ \textit{Case 2}{}: $ $ r\in \sqrt{A} $\ then $ \chi_{\sqrt{A}}=(1,0) $, but by Proposition \ref{pr2.3} $ ,  \sqrt{\chi_A}=\chi_{\sqrt{A}} $\  hence $ \sqrt{\chi_A}=(1,0) $, therefore $ \chi_A(rm)=\sqrt{\chi_A}(r) $.

If $ \chi_A(rm)=(0,1) $ then $ rm\notin A $ and so $ m\notin A $, hence $ \chi_A(m)=(0,1) $. Thus $ \chi_A(rm)=\chi_A(m) $ or $ \chi_A(rm)\leq \sqrt{\chi_A}(r) $.

Conversely, let $ \chi_A $ be an intuitionistic fuzzy primary submodule of $ M $ and suppose that $ rm\in A $, $ r\in R , m\in M $. Then $ \chi_{A}(rm)=~(1,0) $. Now since $ \chi_A $ is an intuitionistic fuzzy primary submodule then
 \begin{center}
 $\chi_A(rm)=\chi_A(m) $\quad or\quad $ \chi_A(rm)\leq \sqrt{\chi_A}(r)=\chi_{\sqrt{A}}(r) $  
 \end{center}
 If $\chi_A(rm)=\chi_A(m) $ then $ \chi_{A}(m)=(1,0) $ and so $ m\in A $. 
 But If $ \chi_A(rm)\leq \chi_{\sqrt{A}}(r) $, then
\begin{equation*}
\begin{array}{ll}
\mu_{\chi_{A}}(rm)\leq \mu_{\chi_{\sqrt{A}}}(r)\ ,\ \nu_{\chi_{A}}(rm)\geq \nu_{\chi_{\sqrt{A}}}(r)&\Longrightarrow \mu_{\chi_{\sqrt{A}}}(r)\geq 1 \ ,\ \nu_{\chi_{\sqrt{A}}}(r)\leq 0\\
&\Longrightarrow \mu_{\chi_{\sqrt{A}}}(r)=1 \ , \ \nu_{\chi_{\sqrt{A}}}(r)=0\\
&\Longrightarrow \chi_{\sqrt{A}}=(1,0)\\
& \Longrightarrow r\in \sqrt{A}
\end{array}
\end{equation*}
\end{proof}
\begin{proposition}
Let $ M $ be an R-module and $ A $ be an  intuitionistic fuzzy primary submodule of  $  M $ with  sup property. Then $ M_{A}^{(\alpha,\beta)} $ is a primary submodule of $ M $.
\end{proposition}
\begin{proof}
Since $ A\in IFM(M)$, so $ M_{A}^{(\alpha,\beta)} $ is a submodule of $ M $. Now let $ r\in R , m\in M $ and $ rm\in M_{A}^{(\alpha,\beta)}  $, then $ \mu_A (rm)\geq \alpha\ ,\ \nu_A (rm)\leq \beta $. Since $ A $ is an intuitionistic fuzzy primary submodule of R-module $ M $ then
$ \alpha\leq \mu_A (rm)= \mu_A(m), \ \beta\geq \nu_A (rm)= \nu_A(m)$ \, or \,
 $ \alpha\leq \mu_A (rm)\leq \sqrt{\mu_A}(r), \ \beta\geq \nu_A (rm)\geq~ \sqrt{\nu_A}(r) $.
Hence
$ \mu_A(m)\geq \alpha \ , \ \nu_A(m)\leq \beta$ \, or \, $\sqrt{\mu_A}(r)\geq \alpha \ , \ \sqrt{\nu_A}(r)\leq \beta $. This mean that $ m\in M_{A}^{(\alpha,\beta)} $\ or \ $ r\in~ (\sqrt{M_A})^{(\alpha ,\beta)} $. But $ A $ has $ sup \ property$ and by Proposition~ \ref{pr3.3} we have $ (\sqrt{M_A})^{(\alpha,\beta)}=\sqrt{M_{A}^{(\alpha,\beta)}} $, therefore $ m\in M_{A}^{(\alpha,\beta)} $  or  $ r\in \sqrt{M_{A}^{(\alpha,\beta)}} $, and it follows that $ M_{A}^{(\alpha,\beta)} $ is a primary submodule of $ M $.
\end{proof}
If $ M_{A}^{(\alpha,\beta)} $ is strict level cut or $ A $ is finite valued intuitionstic fuzzy submodule, then above proposition holds without assumption of sup property.
We know that $ A $ is finite valued intuitionstic fuzzy submodule of $ M $, if 
$ Im \mu=\lbrace \mu(x) : x\in M \rbrace $ , $ Im \nu=\lbrace \nu(x) : x\in M \rbrace $
 be finite.
\begin{proposition}
Let $A$ be an intuitionstic fuzzy submodule of $ M $, such that every level cut of $ A $ is primary submodule of $ M $, then $ A $ is an intuitionistic fuzzy primary submodule of $ M $.
\end{proposition}
\begin{proof}
 Let $ r\in R , m\in M $, and let $ \mu_A(r.m)=\alpha \ , \ \nu_A(r.m)=\beta $, then
\begin{footnotesize}
\begin{align*}
& r.m \in M_{A}^{(\alpha,\beta)} \Longrightarrow m\in M_{A}^{(\alpha,\beta)}\quad or \quad r\in \sqrt{M_{A}^{(\alpha,\beta)}} \Longrightarrow m\in M_{A}^{(\alpha,\beta)}\quad or \quad r^{n}.M \subseteq  M_{A}^{(\alpha,\beta)}\quad for\ some\ n\in \mathbb{N}\\
&\Longrightarrow \mu_A (m)\geq \alpha \ , \ \nu_A (m)\leq \beta \quad or \quad \mu_A (r^{n}.m)\geq \alpha  \ , \ \nu_A (r^{n}.m)\leq \beta \quad for\ all\ m\in M \\
 &\Longrightarrow \alpha=\mu_A (r.m)\geq \mu_A (m)\geq\alpha\ ,\ \beta=\nu_A (r.m)\leq \nu_A (m)\leq\beta\quad or \quad
\bigwedge \limits_{m\in M}\mu_A (r^{n}.m)\geq \alpha \ , \ \bigvee \limits_{m\in M}\nu_A (r^{n}.m)\leq \beta \\
 &\Longrightarrow \mu_A (r.m)= \mu_A (m)\ , \ \nu_A (r.m)= \nu_A (m)\quad or \quad 
\bigvee\limits_{n\in \mathbb{N}}\bigwedge\limits_{m\in M} {\mu_A (r^{n}.m)} \geq \alpha \ , \ \bigwedge\limits_{n\in \mathbb{N}}\bigvee\limits_{m\in M} {\nu_A (r^{n}.m)} \leq \beta \\
 &\Longrightarrow \mu_A (r.m)= \mu_A (m)\ , \ \nu_A (r.m)= \nu_A (m)\quad or \quad \sqrt{\mu_A}(r)\geq\alpha \ , \ \sqrt{\nu_A}(r)\leq \beta \\
 &\Longrightarrow \mu_A (r.m)= \mu_A (m)\ , \ \nu_A (r.m)= \nu_A (m)\quad or \quad \mu_A (r.m)\leq\sqrt{\mu_A}(r) \ , \ \nu_A (r.m)\geq\sqrt{\nu_A}(r)\ . 
\end{align*}
\end{footnotesize}
 This mean that $ A $ is an intuitionistic fuzzy primary submodule of $ M $.
\end{proof}
\begin{definition}
Let $ A $ be an intuitionistic fuzzy submodule of R-module $ M $. The support of $ A $ is defined as $ A^{*}=\lbrace x\in M \vert \mu_A(x)>0 , \nu_A (x)<1\rbrace $.
\end{definition}
In the following we investigated related between intuitionistic fuzzy submodule $ A $ and support of $ A $.
\begin{proposition}
Let $ A $ be an intuitionistic fuzzy primary submodule of R-module $ M $. Then $A^{*}$ is a primary submodule of $ M $.
\end{proposition}
\begin{proof}
Suppose that $ rm\in A^{*} $ for $ r\in R , m\in M $. Then $ \mu_A(x)>0 , \nu_A (x)<1 $. Since $A$ is an intuitionistic fuzzy primary submodule, then

$ \mu_A (rm)=\mu_A(m)\,$or$\, \mu_A(rm)\leq \sqrt{\mu_A}(r)$ , $ \nu_A (rm)=\nu_A(m)\,$or$\, \nu_A(rm)\geq \sqrt{\nu_A}(r)$, thus $ \mu_A(m)>0 \, $or$ \, \sqrt{\mu_A}(r)>0 , \nu_A(m)<1 \,$or$\, \sqrt{\nu_A}(r)<1 $.Then $ m\in A^{*} \,$or$\, r\in \sqrt{A^{*}}$. Therefore we prove that $ A^{*} $ is a primary submodule of $ M $.
\end{proof}
\begin{remark}
The support of an intuitionistic fuzzy submodule may be a primary submodule, but it need not be an intuitionistic fuzzy primary submodule. So the converse of the above proposition need not be true. Note  the following example.
\end{remark}
\begin{example}
Let $R=M=\mathbb{Z}$. Define an intuitionistic fuzzy subset $ A=(\mu_A, \nu_A) $ in $ \mathbb{Z} $ by

$
\mu_A(x)=\begin{cases}
1 & \text{if\ $x\in  6\mathbb{Z} $}\\
0.3 & \text{if\ $x\in  2\mathbb{Z}-6\mathbb{Z} $}\\
0 & \text{if\ $x\in  \mathbb{Z}-2\mathbb{Z} $}
\end{cases}\quad  ,\qquad
\nu_A(x)=\begin{cases}
0 & \text{if\ $x\in  6\mathbb{Z} $}\\
0.5 & \text{if\ $x\in 2\mathbb{Z}-6\mathbb{Z}$}\\
1 & \text{if\ $x\in \mathbb{Z}-2\mathbb{Z}$}
\end{cases}
$

It is clear that $ A $ is an intuitionistic fuzzy submodule of $ \mathbb{Z}$ and $ A^{*}=2\mathbb{Z}$,  which is a primary submodule of $ \mathbb{Z} $. we show that $ A $ is not an intuitionistic fuzzy primary submodule of $ \mathbb{Z} $.
Consider $ r=2 , m=3 $. Then $ \mu_A(rm)=\mu_A(2.3)=\mu_A(6)=1\neq \mu_A(3)=0 \, , \, \nu_A(rm)=\nu_A(2.3)=\nu_A(6)=0\neq \nu_A(3)=1 $ and $ \mu_A(rm)=\mu_A(2.3)>\sqrt{\mu_A}(2)=\bigvee\limits_{n\in \mathbb{N}}\bigwedge\limits_{m\in \mathbb{Z}} {\mu_A (2^{n}.m)}=0.3 \, , \, \nu_A(rm)=\nu_A(2.3)<\sqrt{\nu_A}(2)=\bigwedge\limits_{n\in \mathbb{N}}\bigvee\limits_{m\in \mathbb{Z}} {\nu_A (2^{n}.m)}=0.5 $. Therefore $ A $ is not an intuitionistic fuzzy primary submodule of $ \mathbb{Z} $.
\end{example}

\begin{definition} \cite{5}
An intuitionistic fuzzy ideal $ A $ is said to be intuitionistic fuzzy weakly primary ideal if for any $ x,y \in R $
\begin{equation} \label{eq6}
\mu_A(x.y)=\mu_A(x) \ , \ \nu_A(x.y)=\nu_A(x)\quad or \quad \mu_A(x.y)\leq \mu_A(y^{n})\ , \ \nu_A(x.y)\geq \nu_A(y^{n})
\end{equation}
for some $n\in \mathbb{N} $
\end{definition}
\begin{proposition}
Let $ A $ be an intuitionistic fuzzy primary submodule of $ M $ with $ sup\ property$, then $ \bar{A} $
is an intuitionistic fuzzy weakly primary ideal of $ R $ and $ \sqrt{A} $ is the intuitionistic fuzzy prime ideal of $ R $.
\end{proposition}
\begin{proof}
Let $ r_1 , r_2\in R $, then

$ \mu_{\bar{A}}(r_1 r_2)= \bigwedge\limits_{m\in M}\mu_A(r_1r_2.m)=\bigwedge\limits_{m\in M}\mu_A(r_1(r_2.m)) $ , $ \nu_{\bar{A}}(r_1 r_2)=\bigvee\limits_{m\in M}\nu_A(r_1r_2.m)=
\bigvee\limits_{m\in M}\nu_A(r_1(r_2.m)) $.
Now since $ A $ is an intuitionistic fuzzy primary submodule, then
\begin{align*}
&\mu_A(r_1(r_2.m))=\mu_A(r_1.m)\quad or \quad \mu_A(r_1(r_2.m))\leq \sqrt{\mu_A}(r_2) \Longrightarrow \\
&\Longrightarrow\mu_{\bar{A}}(r_1 r_2)= \bigwedge\limits_{m\in M}\mu_A(r_1.m)\quad or \quad \mu_{\bar{A}}(r_1 r_2)\leq \bigwedge\limits_{m\in M}\sqrt{\mu_A}(r_2)=\sqrt{\mu_A}(r_2)\\
&\Longrightarrow \mu_{\bar{A}}(r_1r_2)=\mu_{\bar{A}}(r_1)\quad or \quad \mu_{\bar{A}}(r_1r_2)\leq \sqrt{\mu_A}(r_2)=\bigvee\limits_{n\in \mathbb{N}}\bigwedge\limits_{m\in M}\mu_A(r_{2}^{n}.m )\\
&\Longrightarrow\mu_{\bar{A}}(r_1r_2)=\mu_{\bar{A}}(r_1)\quad or \quad \mu_{\bar{A}}(r_1r_2)\leq \bigwedge\limits_{m\in M}\mu_A(r_{2}^{n_2}.m )\\ 
&for\ some\  n_2 \in \mathbb{N} 
(Because\ A\ has\ sup\ property)\\
&\Longrightarrow\mu_{\bar{A}}(r_1r_2)=\mu_{\bar{A}}(r_1)\quad or \quad \mu_{\bar{A}}(r_1r_2)\leq \bigwedge\limits_{m\in M}\mu_A(r_{2}^{n_2}.m )\leq \bigwedge\limits_{m\in M}\mu_A(r_{2}^{n_1}r_{2}^{n_2}.m )\\ 
&for\ some\ {n_1 , n_2}\in\mathbb{N}
\end{align*}
and also
\begin{align*}
&\nu_A(r_1(r_2.m))=\nu_A(r_1.m)\quad or \quad \nu_A(r_1(r_2.m))\geq \sqrt{\nu_A}(r_2)\Longrightarrow\\
&\Longrightarrow\nu_{\bar{A}}(r_1 r_2)= \bigvee\limits_{m\in M}\nu_A(r_1.m)\quad or \quad \nu_{\bar{A}}(r_1 r_2)\geq \bigvee\limits_{m\in M}\sqrt{\nu_A}(r_2)=\sqrt{\nu_A}(r_2)\\
&\Longrightarrow \nu_{\bar{A}}(r_1r_2)=\nu_{\bar{A}}(r_1)\quad or \quad \nu_{\bar{A}}(r_1r_2)\geq \sqrt{\nu_A}(r_2)=\bigwedge\limits_{n\in \mathbb{N}}\bigvee\limits_{m\in M}\nu_A(r_{2}^{n}.m )\\
&\Longrightarrow\nu_{\bar{A}}(r_1r_2)=\nu_{\bar{A}}(r_1)\quad or \quad \nu_{\bar{A}}(r_1r_2)\geq \bigvee\limits_{m\in M}\nu_A(r_{2}^{n_1}.m )\\
 &for\ some\  n_1 \in \mathbb{N}\: (Because\ A\ has\ sup\ property)\\
&\Longrightarrow\nu_{\bar{A}}(r_1r_2)=\nu_{\bar{A}}(r_1)\quad or \quad \nu_{\bar{A}}(r_1r_2)\geq \bigvee\limits_{m\in M}\nu_A(r_{2}^{n_1}.m )\geq \bigvee\limits_{m\in M}\nu_A(r_{2}^{n_1}r_{2}^{n_2}.m )\\   
\end{align*}
for some $ n_1 , n_2\in\mathbb{N}$. If we consider  $n'=n_1+n_2$  then we have
 
$\mu_{\bar{A}}(r_1r_2)=\mu_{\bar{A}}(r_1)\ ,\ \nu_{\bar{A}}(r_1r_2)=\nu_{\bar{A}}(r_1) \quad or \quad  \mu_{\bar{A}}(r_1r_2)\leq \mu_{\bar{A}}(r_1^{n'})\ , \ \nu_{\bar{A}}(r_1r_2)\geq \nu_{\bar{A}}(r_1^{n'})$ .

This proves that $ \bar{A} $ is an intuitionistic fuzzy weakly primary ideal of $ R $.
Next, by remark \ref{re1.3} we have $ \sqrt{\bar{A}}=\sqrt{A} $ and since $ \sqrt{\bar{A}}$ is intuitionistic fuzzy  prime ideal then $\sqrt{A} $ is intuitionistic fuzzy  prime ideal of $R$.
\end{proof}
The opposite of the above proposition is true if $A$ is an ideal.
\begin{theorem}
Let $ A\in IFI(R)$ with $ sup \ property$. Then $ A $ is an intuitionistic fuzzy primary submodule if and only if $ A $ is an intuitionistic fuzzy weakly primary ideal of $ R $.
\end{theorem}
\begin{proof}
Suppose that $ A $ is an intuitionistic fuzzy primary submodule of $ R $, then clearly $ A $ is intuitionistic fuzzy primary ideal  of $ R $.
Now since $ A $ is an intuitionistic fuzzy primary submodule, then for $ r_1 , r_2 \in R $

$ \mu_A(r_1r_2)=\mu_A(r_1)\ , \ \nu_A(r_1r_2)=\nu_A(r_1) $\; or \; $\mu_A(r_1r_2)\leq \sqrt{\mu_A}(r_2)\ , \ \nu_A(r_1r_2)\geq \sqrt{\nu_A}(r_2)\Longrightarrow$
\begin{equation*}
\begin{array}{ll}
 \Longrightarrow\mu_A(r_1r_2)\leq \bigvee\limits_{n\in \mathbb{N}}\mu_A(r_2^n)\ , \ \nu_A(r_1r_2)\geq \bigwedge\limits_{n\in \mathbb{N}}\nu_A(r_2^n)\\
\Longrightarrow \mu_A(r_1r_2)\leq \mu_A(r_2^{n_1})\ ,\ \nu_A(r_1r_2)\geq \nu_A(r_2^{n_2})\, $for some$\  n_1 , n_2 \in \mathbb{N}\,
\textsc{($Since $ A $ has\  Sup\  property$)}\\
\Longrightarrow \mu_A(r_1r_2)\leq \mu_A(r_2^{n_1})\leq \mu_A(r_2^{n_1}r_2^{n_2})\ , \ \nu_A(r_1r_2)\geq \nu_A(r_2^{n_2})\geq \nu_A(r_2^{n_1}r_2^{n_2})\\
\Longrightarrow \mu_A(r_1r_2)\leq \mu_A(r_2^{n'})\ ,\ \nu_A(r_1r_2)\geq \nu_A(r_2^{n'})\qquad  ($consider$\ n_1+n_2=n')\ .   
\end{array}
\end{equation*}
Therefore we obtain that

$ \mu_A(r_1r_2)=\mu_A(r_1)\ , \ \nu_A(r_1r_2)=\nu_A(r_1) $\quad or \quad $ \mu_A(r_1r_2)\leq \mu_A(r_2^{n'})\ ,\ \nu_A(r_1r_2)\geq \nu_A(r_2^{n'}) $.

Hence $ A $ is an intuitionistic fuzzy weakly primary ideal of $ R $.

Conversely, let $ A $ be an intuitionistic fuzzy weakly primary ideal of $ R $, then

 $  \mu_A(r_1r_2)=\mu_A(r_1)\ , \ \nu_A(r_1r_2)=\nu_A(r_1) $\: or \: $ \mu_A(r_1r_2)\leq \mu_A(r_2^{n})\ ,\ \nu_A(r_1r_2)\geq \nu_A(r_2^{n}) $\ for some $n \in \mathbb{N}$.
 Therefore $  \mu_A(r_1r_2)=\mu_A(r_1)\ , \ \nu_A(r_1r_2)=\nu_A(r_1) $\ or \ $ \mu_A(r_1r_2)\leq \bigvee\limits_{n \in \mathbb{N}} \mu_A(r_2^{n})\ ,\ \nu_A(r_1r_2)\geq \bigwedge\limits_{n \in \mathbb{N}} \nu_A(r_2^{n})$.  
Thus $  \mu_A(r_1r_2)=\mu_A(r_1)\ , \ \nu_A(r_1r_2)=\nu_A(r_1) $\: or \: $ \mu_A(r_1r_2)\leq \sqrt{\mu_A}(r_2)\ ,\ \nu_A(r_1r_2)\geq \sqrt{\nu_A}(r_2) $ for all $ r_1 , r_2 \in R $.
Hence $ A $ is intuitionistic fuzzy primary submodule of $ R $.
\end{proof}
The above theorem shows that the concept of the intuitionistic fuzzy primary submodule is the generalization of the intuitionistic fuzzy primary ideal.
\begin{example}
Let $ \mathbb{Z}_{12} $ as $ \mathbb{Z}-module $. Define $ A=\lbrace m,\mu_A(m),\nu_A(m)\rbrace $ such that $ m\in \mathbb{Z}_{12} $. By 

\begin{small}
$
\begin{cases} 
&\mu_A(m)=1\\
&\nu_A(m)=0

\end{cases}
$
 \begin{scriptsize}
$ if\ m\in \lbrace \bar{0},\bar{4},\bar{8}\rbrace$
\end{scriptsize}   and 
$
\begin{cases}
&\mu_A(m)=\frac{1}{2}\\
&\nu_A(m)=\frac{1}{2}
\end{cases}
$
\begin{scriptsize}
$if\ m\in \lbrace \bar{2},\bar{6},\bar{10}\rbrace$
\end{scriptsize} and 
$
\begin{cases} 
&\mu_A(m)=0\\
&\nu_A(m)=1
\end{cases}
$
 \begin{scriptsize}
$if\ m\in~ \lbrace \bar{1},\bar{3},\bar{5},\bar{7},\bar{9},\bar{11}\rbrace$
\end{scriptsize} 
$
$
\end{small}
We get that $ A $ is an intuitionistic fuzzy submodule of $ \mathbb{Z}_{12} $.
Now $ \bar{A}=(x,\mu_{\bar{A}}(x),\nu_{\bar{A}}(x)) $ for $ x\in \mathbb{Z} $ defined by:\\

$
\begin{cases}
&\mu_{\bar{A}}(z)= \bigwedge\limits_{m\in \mathbb{Z}_{12}}\mu_A(z.m)=1\\
&\nu_{\bar{A}}(z)= \bigvee\limits_{m\in \mathbb{Z}_{12}}\nu_A(z.m)=0
\end{cases}
$
\quad  if $ z\in <\bar{4}> $ \\\\
 
$
\begin{cases}
&\mu_{\bar{A}}(z)= \bigwedge\limits_{m\in \mathbb{Z}_{12}}\mu_A(z.m)=\frac{1}{2}\\
&\nu_{\bar{A}}(z)= \bigvee\limits_{m\in \mathbb{Z}_{12}}\nu_A(z.m)=\frac{1}{2}
\end{cases}
$
\quad  if $ z\in <\bar{2}>-<\bar{4}> $ \\\\

$
\begin{cases}
&\mu_{\bar{A}}(z)= \bigwedge\limits_{m\in \mathbb{Z}_{12}}\mu_A(z.m)=0\\
&\nu_{\bar{A}}(z)= \bigvee\limits_{m\in \mathbb{Z}_{12}}\nu_A(z.m)=1
\end{cases}
$
\quad otherwise\\

Clearly $ \bar{A} $ is an intuitionistic fuzzy ideal of $ \mathbb{Z} $. Now we know that $ \sqrt{\bar{A}}=\sqrt{A} $ defined by  $ \sqrt{A}=(\sqrt{\mu_A},\sqrt{\nu_A}) $ and\\
 
 $
 \begin{cases}
 &\sqrt{\mu_A}(z)=1\\
 &\sqrt{\nu_A}(z)=0
 \end{cases}
 $
 \quad if $ z\in <\bar{2}> $ \qquad and \quad
 $
\begin{cases}
 &\sqrt{\mu_A}(z)=0\\
 &\sqrt{\nu_A}(z)=1
 \end{cases}
 $
 \quad otherwise.  \\
 
  It is easy to show that $ A(z.m)=A(m)\,or\, A(z.m)\leq \sqrt{A}(z) $ for all $ m\in {\mathbb{Z}_{12}} , z\in \mathbb{Z} $. Hence $ A $ is an intuitionistic fuzzy primary submodule of $ {\mathbb{Z}_{12}} $.
\end{example}
\begin{proposition}
Let $ f:M\longrightarrow M' $ be an epimorphism of R-module $ M $ to R-module $ M'$. If $ A\in IFM(M)$, then $ \sqrt{A}\subseteq \sqrt{f(A)} $. Equality  holds if $ A $ is constant on kernel 
$ f $.
\end{proposition}
\begin{proof}
By Theorem \ref{th1.3} we have $ \bar{A}\subseteq \overline{f(A)} $, therefore $ \sqrt{\bar{A}}\subseteq \sqrt{\overline{f(A)}} $ and hence $ \sqrt{A}\subseteq \sqrt{f(A)} $. Also $ \bar{A}= \overline{f(A)} $ if $ A $ is constant on kernel $ f $, hence $ \sqrt{\bar{A}}= \sqrt{\overline{f(A)}} $ and therefore $ \sqrt{A}= \sqrt{f(A)} $.
\end{proof}
\begin{proposition}
Let $ f:M\longrightarrow M' $ be a homomorphism of R-module $ M $ to R-module $ M' $. If $ B\in IFM(M')$, then $ \sqrt{f^{-1}(B)}\supseteq \sqrt{B} $. Equality  holds if $ f $ is epimorphism.
\end{proposition}
\begin{proof}
By Theorem \ref{th2.3}, $ \overline{f^{-1}(B)}\supseteq \bar{B} $, hence $ \sqrt{\overline{f^{-1}(A')}}\supseteq \sqrt{\bar{B}} $ and therefore $ \sqrt{f^{-1}(B)}\supseteq \sqrt{B} $. If $ f $ is epimorphism,   then $ \overline{f^{-1}(B)}= \bar{B} $ and hence $ \sqrt{f^{-1}(B)}= \sqrt{B} $.
\end{proof}
\begin{theorem}
Let $ f:M\longrightarrow M' $ be an epimorphism of R-module $ M $ to R-module $ M' $, and let $ A $ be an intuitionistic fuzzy primary submodule of $ M $, which is constant on kernel $ f $, then the image $f$, $f(A)$ is an intuitionistic fuzzy primary submodule of $ M' $. 
\end{theorem}
\begin{proof}
By Theorem \ref{th1.3}  $ f(A) $ is an intuitionistic fuzzy submodule of $ M'$. Let $ r\in R\,, m'\in M'$, then
 
\begin{center}
$ f(\mu_A)(rm')=\bigvee\limits_{m\in f^{-1}(rm')}\mu_A(m)\ , \  f(\nu_A)(rm')=\bigwedge\limits_{m\in f^{-1}(rm')}\mu_A(m)$
\end{center}
Since $ A $ constant on kernel $ f $ then
\begin{equation*}
\begin{array}{ll}
$if$\;  m\in f^{-1}(rm')&\Longrightarrow f(m)=rm'=rf(m_1) \quad ($There is $  m_1\in M $ such that $ f(m_1)=m' )\\
& \Longrightarrow f(m)-f(rm_1)=0\\ 
 &\Longrightarrow m-rm_1 \in ker f\\
 & \Longrightarrow \mu_A(m)=\mu_A(rm_1)\ , \ \nu_A(m)=\nu_A(rm_1).
\end{array}
\end{equation*}
 Therefore $ f(\mu_A)(rm')=\bigvee\limits_{m_1\in f^{-1}(m')}\mu_A(rm_1)\ , \  f(\nu_A)(rm')=\bigwedge\limits_{m_1\in f^{-1}(m')}\mu_A(rm_1)$, but $ A $ is an intuitionistic fuzzy primary submodule of $ M $ and hence
 
 $ \mu_A(rm_1)=\mu_A(m_1)\ , \ \nu_A(rm_1)=\nu_A(m_1) $\quad or \quad $ \mu_A(rm_1)\leq \sqrt{\mu_A}(r)\ ,\ \nu_A(rm_1)\geq \sqrt{\nu_A}(r)$.\\ 
 Now if  $ \mu_A(rm_1)=\mu_A(m_1)\ , \ \nu_A(rm_1)=\nu_A(m_1) $ then
 
\begin{small} 
$ f(\mu_A)(rm')=\bigvee\limits_{m_1\in f^{-1}(m')}\mu_A(m_1)=f(\mu_A)(m')\quad ; \quad
 f(\nu_A)(rm')=\bigwedge\limits_{m_1\in f^{-1}(m')}\mu_A(m_1)= f(\nu_A)(m')$.
 
And if $ \mu_A(rm_1)\leq \sqrt{\mu_A}(r)\ ,\ \nu_A(rm_1)\geq \sqrt{\nu_A}(r)$ then

\begin{footnotesize}
$\bigvee\limits_{m_1\in f^{-1}(m')}\mu_A(rm_1)\leq \bigvee\limits_{m_1\in f^{-1}(m')}\sqrt{\mu_A}(r)\Longrightarrow f(\mu_A)(rm')\leq \bigvee\limits_{m_1\in f^{-1}(m')}\sqrt{\mu_A}(r)=\sqrt{\mu_A}(r)\leq \sqrt{f(\mu_A)}(r)$\\

and $ \bigwedge\limits_{m_1\in f^{-1}(m')}\nu_A(rm_1)\geq \bigwedge\limits_{m_1\in f^{-1}(m')}\sqrt{\nu_A}(r)\Longrightarrow f(\nu_A)(rm')\geq \bigwedge\limits_{m_1\in f^{-1}(m')}\sqrt{\nu_A}(r)=\sqrt{\nu_A}(r)\geq \sqrt{f(\nu_A)}(r)$
\end{footnotesize}
\end{small}
 
Therefore we got that
$ f(\mu_A)(rm')=f(\mu_A)(m')\ , \ f(\nu_A)(rm')=f(\nu_A)(m')$\ or \ $ f(\mu_A)(rm')\leq \sqrt{f(\mu_A)}(r)\ ,\ f(\nu_A)(rm')\geq~ \sqrt{f(\nu_A)}(r)$.
This show that $ f(A) $ is an intuitionistic fuzzy primary submodule of $ M' $.
\end{proof}
\begin{theorem}
Let $ f:M\longrightarrow M' $ be a homomorphism of R-module $ M $ to R-module $ M'$. If $ B $ is an intuitionistic fuzzy primary submodule of $ M' $, then $ f^{-1}(B) $ is an intuitionistic fuzzy primary submodule of $ M $.
\end{theorem}
\begin{proof}
Clearly $ f^{-1}(B) $ is an intuitionistic fuzzy submodule of $ M $. Let $ r\in R\ ,\ m\in M $, then

$ f^{-1}(\mu_{B})(rm)=\mu_{B}(f(rm))=\mu_{B}(rf(m))\ , \ f^{-1}(\nu_{B})(rm)=\nu_{B}(f(rm))=\nu_{B}(rf(m)) $.

But $B$ is an intuitionistic fuzzy primary submodule, hence

$ \mu_{B}(rf(m))=\mu_{B}(f(m))=f^{-1}(\mu_{B})(m) \ , \ \nu_{B}(rf(m))=\nu_{B}(f(m))=f^{-1}(\nu_{B})(m) $\quad or

$ \mu_{B}(rf(m))\leq\sqrt{\mu_{B}}(r)\leq\sqrt{f^{-1}(\mu_{B})}(r)\ , \ \nu_{B}(rf(m))\geq\sqrt{\nu_{B}}(r)\geq\sqrt{f^{-1}(\nu_{B})}(r) $.

Thus we show that

$ f^{-1}(\mu_{B})(rm)=f^{-1}(\mu_{B})(m)\ , \ f^{-1}(\nu_{B})(rm)=f^{-1}(\nu_{B})(m) $\quad or 

$ f^{-1}(\mu_{B})(rm)\leq\sqrt{f^{-1}(\mu_{B})}(r)\ , \ f^{-1}(\nu_{B})(rm)\geq\sqrt{f^{-1}(\nu_{B})}(r) $.

This prove that$ f^{-1}(B) $ is an intuitionistic fuzzy primary submodule of $ M $.\\
\end{proof}
\begin{conclusion}
In this paper, we defined the radical of an intuitionistic fuzzy submodule, and then we studied the intuitionistic fuzzy primary submodules with the help of this concept. We also presented some properties of homomorphic image and pre-image of these submodules.
\end{conclusion}

\vspace{0.1in}
\hrule width \hsize \kern 1mm
\end{document}